\documentclass[aap]{imsart}

\RequirePackage{amsthm,amsmath,amsfonts,amssymb}
\RequirePackage[numbers,sort&compress]{natbib}
\RequirePackage[colorlinks,citecolor=blue,urlcolor=blue]{hyperref}
\RequirePackage{graphicx}

\startlocaldefs
\theoremstyle{plain}

\newtheorem{theorem}{Theorem}[section]
\newtheorem{lemma}[theorem]{Lemma}
\newtheorem{prop}[theorem]{Proposition}
\theoremstyle{remark}
\newtheorem{definition}[theorem]{Definition}
\newtheorem{rem}[theorem]{Remark}

\def \a{\alpha}   \def \d{\delta}
   \def \e{\epsilon}
\def \s{\sigma} \def \l{\lambda}  \def \o{\omega}
\def \O{\Omega} \def \D{\Delta}

\def\leq{\leqslant}

\def\tilde{\widetilde}
\def\hat{\widehat}

\endlocaldefs

\begin{document}

\begin{frontmatter}
\title{Finite ergodic components  for upper probabilities}
%\title{A sample article title with some additional note\thanksref{t1}}
\runtitle{Finite ergodic components  for upper probabilities}
%\thankstext{T1}{A sample additional note to the title.}

\begin{aug}
\author[A,D]{\fnms{Chunrong}~\snm{Feng}\ead[label=e1]{chunrong.feng@durham.ac.uk}\orcid{0000-0001-5244-4664}},
\author[B]{\fnms{Wen}~\snm{Huang}\ead[label=e2]{wenh@mail.ustc.edu.cn}\orcid{0000-0003-1074-1331}}
\author[C,A,B]{\fnms{Chunlin}~\snm{Liu}\ead[label=e3]{chunlinliu@mail.ustc.edu.cn} \orcid{0000-0001-6277-013X}}

\author[A,D]{\fnms{Huaizhong}~\snm{Zhao}\ead[label=e4]{huaizhong.zhao@durham.ac.uk}\orcid{0000-0002-8873-2040}}
%%%%%%%%%%%%%%%%%%%%%%%%%%%%%%%%%%%%%%%%%%%%%%
%% Addresses                                %%
%%%%%%%%%%%%%%%%%%%%%%%%%%%%%%%%%%%%%%%%%%%%%%
\address[A]{Department of Mathematical Sciences, Durham University, DH1 3LE, United Kingdom
	%\printead[presep={,\ }]
	%{e1,e4}
}

\address[B]{School of Mathematical Sciences, University of Science and Technology of China, Hefei, Anhui, 230026, P.R. China
	%\printead[presep={,\ }]
	%{e2,e3}
}
\address[C]{School of Mathematical Sciences, Dalian University of Technology, Dalian, 116024, P.R. China
	%\printead[presep={,\ }]
	%{e2,e3}
}
\address[D]{Research Center for Mathematics and Interdisciplinary Sciences, Shandong University, Qingdao 266237, China 
}
\address{
	\printead[presep={\ }]{e1,e2,e3,e4}
}
\end{aug}

\begin{abstract}
Under the notion of ergodicity of upper probability in the sense of Feng and Zhao (2021) that  any invariant set either has capacity $0$ or its complement has capacity 0, we introduce the definition of finite ergodic components (FEC). We prove an invariant upper probability has FEC if and only if  it is in the regime that any invariant set has either capacity $0$ or capacity $1$, proposed by 	Cerreia-Vioglio, Maccheroni, and Marinacci (2016). Furthermore, this is also equivalent to that the eigenvalue $1$ of the Koopman operator is of finite multiplicity, while  in the  ergodic upper probability regime, as in the classical ergodic probability case, the  eigenvalue $1$ of the Koopman operator is simple.

Additionally, we obtain the equivalence of the law of large numbers with multiple values, the asymptotic independence and the FEC. Furthermore, we apply these to obtain the corresponding results for non-invariant probabilities.
\end{abstract}


\end{frontmatter}

	\section{Introduction}
Let $(\O,\mathcal{F})$ be a measurable space and $\D(\O,\mathcal{F})$ be the space consisting of all finitely additive probabilities on $(\O,\mathcal{F})$.  Recall a set-valued function $\mu: \mathcal{F} \rightarrow[0,1]$ is
a capacity if $\mu(\emptyset)=0, \mu(\Omega)=1$, and $\mu(A) \leq \mu(B) $ for all $A, B \in \mathcal{F}$ such that $A \subset B$. For a  capacity $\mu$, we denote its  core by  
$$
\operatorname{core}(\mu)=\{P\in\D(\O,\mathcal{F}):P(A)\le \mu(A)\text{ for any }A\in\mathcal{F}\}.
$$
In this paper, we focus on a class of subadditive capacities, that is, upper probabilities. More precisely, a capacity $V$ on $(\O,\mathcal{F})$ is  said to be an upper probability if $V=\sup_{P\in\operatorname{core}(V)}P$ and $\lim_{n \rightarrow \infty}V(A_n)=0$ for any $A_n\in\mathcal{F}$ with $A_n\downarrow \emptyset$. It is easy to see that when $V$ is an upper probability then each element in $\operatorname{core}(V)$ is a probability.

If an upper probability $V$ is invariant with respect to a measurable transformation $T:\O\to\O$, i.e., $V(T^{-1}A)=V(A)$ for any $A\in\mathcal{F}$, then $(\O,\mathcal{F},V,T)$ is called an upper probability preserving system (UPPS for short), where $T^{-1}A:=\{\o\in \O:T\o\in A\}$.
In order to achieve the irreducibility of an UPPS  $(\O,\mathcal{F},V,T)$, following the ideas by Feng and Zhao \cite{FengZhao2021} in the realm of sublinear expectations, Feng, Wu and Zhao \cite{FWZ2020} gave the corresponding definition of ergodicity, 
namely, an invariant upper probability $V$ is said to be ergodic if for any $A\in\mathcal{I}$, one has
\[\text{ $V(A)=0$ or $V(A^c)=0$,}\]
where $\mathcal{I}:=\{A\in\mathcal{F}:T^{-1}A=A\}$. 
Recently, Feng, Huang, Liu and Zhao \cite{FHLZ}, proved that for an ergodic upper probability $V$, there exists a unique invariant probability $Q$ in its core,  and this probability $Q$ is ergodic in the classical sense. Meanwhile, it was proven that $V$ is absolutely continuous with respect to $Q$ (see Lemma \ref{lem:core of invariant lower probabilities} for details). 

In the following (see Lemma \ref{lem:finitely many probability}), we will see that for any invariant upper probability, there are at most finitely many ergodic probabilities in its core. With this result and  an invariant probability in the core of an ergodic upper probability result in \cite{FHLZ}, it is then natural to consider a situation that an invariant upper probability $V$ can be decomposed into finitely many ergodic components in the following sense.
\begin{definition}
	Let $(\O,\mathcal{F},V,T)$ be an UPPS. The upper probability $V$ is said to be of finite ergodic components if there exists a partition $\a=\{A_1,\ldots,A_n\}\subset \mathcal{I}$, such that for each $i\in\{1,2,\ldots,n\}$, $V_i$  is an ergodic upper probability with respect to $T$ on $(\O,\mathcal{F})$, where 
	\begin{equation}\label{eq:maindeifnition}
		V_i(A):=\sup\{P(A):P\in\operatorname{core}(V)\text{ and } P(A_i)=1\}\text{ for any }A\in\mathcal{F}.
	\end{equation}
\end{definition} 
\begin{rem}
	For each $i\in\{1,2,\ldots,n\}$,  $V_i$ can be viewed as an ergodic upper probability on $(A_i,\mathcal{F}_i)$, where $\mathcal{F}_i=\{A\in\mathcal{F}:A\subset A_i\}$, as $V_i(A_i^c)=0$, $i=1,2,\ldots,n$.
\end{rem}

In the following, we denote by $\mathcal{ M}(T)$ the set of all invariant probabilities on $(\O,\mathcal{F})$ and $\mathcal{ M}^e(T)$ the set of all ergodic probabilities on $(\O,\mathcal{F})$. In the classical ergodic theory, it is well-known that  every invariant probability has an ergodic decomposition  (see \cite[Theorem 4.8]{Ward2011} for example). In particular, if $\mathcal{ M}^e(T)$  has only finitely many elements, each invariant probability is a  unique convex combination of ergodic probabilities.
In this paper, we will prove each upper probability with finite ergodic components has a local ergodic decomposition (see Theorem \ref{thm:1}). Specifically, let $V$ be an invariant upper probability, and denote  the set $\mathcal{ M}^e(T)\cap\operatorname{core}(V)=\{Q_1,\ldots,Q_n\}$\footnote{By Lemma \ref{lem:finitely many probability}, its cardinality must be finite.}.  Then $V$ is of finite ergodic components if and only if  
for any invariant probability $P\in \operatorname{core}(V)$  there exist $0\le \a_1,\ldots,\a_{n}\le 1$ with $\sum_{i=1}^{n}\a_i=1$  such that 
\[P=\a_1Q_1+\ldots+\a_nQ_n,\] where this decomposition is unique. In fact, we will provide a decomposition for all invariant upper probabilities, without the requirement that they are of finite ergodic components (see Theorem  \ref{thm:decompostion for all invariant upps}).

Cerreia-Vioglio, Maccheroni, and Marinacci \cite{CMM2016} called a capacity $\nu$  ergodic if $\nu(\mathcal{I})=\{0,1\}$, in parallel to the definition of the invariant measure $P$ being ergodic if $P(\mathcal{I})=\{0,1\}$. They were mainly interested in the case of a lower probability $\nu$. It is noted, due to the duality of $\nu$ and an upper probability $V$, that $\nu(\mathcal{I})=\{0,1\}$ is equivalent to $V(\mathcal{I})=\{0,1\}$. Thus in the sense of \cite{CMM2016}, an invariant upper probability $V$ was defined as ergodic if $V(\mathcal{I})=\{0,1\}$. 

It is easy to see that the $V(\mathcal{I})=\{0,1\}$ condition is weaker than the concept of ergodicity introduced in  \cite{FengZhao2021}. 
The important point is that the $\{0,1\}$ condition does not guarantee the indecomposability  of dynamical systems, in contrast to the classical ergodic theory of $\sigma$-additive probabilities.  However, it was proved in  \cite{FengZhao2021, FWZ2020} that the ergodicity in the sense of  \cite{FengZhao2021} is equivalent to that the dynamical system is not decomposable, and also equivalent to the eigenvalue $1$ of the Koopman transformation operator is simple. All these important features are the same as in the classical invariant probability case.

Although the condition $V(\mathcal{I})=\{0,1\}$  cannot prevent  an UPPS  to be decomposed into smaller ergodic components, it still has many interesting implications. For example,  the authors \cite{CMM2016} proved the existence of each Birkhoff's sum and provided upper and lower bounds for it under the condition that $V(\mathcal{I})=\{0,1\}$ (see also \cite{sheng2023continuous} by  Sheng and Song).
In this paper, we will prove that $V(\mathcal{I})=\{0,1\}$  is equivalent to $V$ being of finite ergodic components, and also equivalent to that Birkhoff's ergodic theorem holds with finite multiple values. Moreover, each component $(A_i,\mathcal{F}_i,V_i)$ is ergodic in the sense of \cite{FengZhao2021} (see Theorem \ref{thm:1}). As a corollary, we provide a concrete value for each Birkhoff's sum. This is a crucial result as we notice that without knowing the concrete (multi-)values, it would be impossible to get the above equivalence and many other results such as the eigenvalue multiplicity result mentioned below.
In this connection, some other characterizations of finite ergodic components by invariant functions and independence are also provided (see Sections \ref{sec: invariant functions} and \ref{sec:independence}). As a consequence, the finite ergodic components condition is equivalent to that the eigenvalue $1$ of the Koopman operator is of finite multiplicity. These new results give equivalent descriptions of an UPPS in the regime $V(\mathcal{I})=\{0,1\}$,  in terms of its dynamics, and equilibria in the statistical sense and the analytic sense, respectively. They are in a clear contrast with, but strongly related to, that 
of the  ergodicity in the sense of \cite{FengZhao2021}.

An important application  of  nonlinear probability theory is to study  the Birkhoff's ergodic theorem  for  non-invariant probabilities,  as observed in \cite{FHLZ}, namely, for a probability space $(\O,\mathcal{F},P)$ and an invertible measurable map $T:\O\to\O$. The program of studying an ergodic theory for non-invariant probabilities was initiated by Hurewicz \cite{Hurewicz1944}. In the authors' previous article \cite{FHLZ}, we utilized nonlinear methods to further investigate the results presented in \cite{Hurewicz1944}.
As an application of the nonlinear results from this paper, we derive a version of Birkhoff's ergodic theorem for a broader class of non-invariant probabilities (see Theorem \ref{thm:noninvariant}). Namely, let $(\O,\mathcal{F},P)$ be a probability space  and $T:\O\to \O$ be a measurable invertible transformation (where $P$ does not need to be invariant). If the limit 
$\lim_{N\to\infty}\frac{1}{N}\sum_{i=0}^{N-1}P(T^{-i}A)$ exists for any $A\in\mathcal{F}$, then the following statements are equivalent:
\begin{enumerate}
	\item [(i)] $\#(P(\mathcal{I}))<\infty$;
	\medskip
	\item [(ii)] there exist $Q_1,\ldots,Q_n\in\mathcal{M}^e(T)$ and  a  partition $\a=\{A_1,\ldots,A_n\}\subset\mathcal{I}$ with  $P(A_i)>0$, $i=1,2,\ldots,n$, such that for any bounded $\mathcal{F}$-measurable function $f$, 
	\begin{equation*}
		\lim_{n\to\infty}\frac{1}{n}\sum_{i=0}^{n-1}f\circ T^i=\sum_{j=1}^{n}\left(\int fdQ_j\right)1_{A_j},  ~~P\text{-a.s.}
	\end{equation*}
\end{enumerate}
In particular, if $\#(P(\mathcal{I}))=2$, i.e., $P(\mathcal{I})=\{0,1\}$, then we  obtain the corresponding result in \cite{FHLZ}. 
\medskip

\medskip

The structure of the paper is as follows. In Section \ref{sec:prelimilaries}, we review some basic notions and establish fundamental properties that will be used throughout the paper. In Section \ref{sec:ergodic decomposition}, we prove a local ergodic decomposition of invariant upper probabilities. As corollaries, we obtain the equivalence between finite ergodic components and $V(\mathcal{I})=\{0,1\}$, as well as a version of Birkhoff's ergodic theorem. In Sections \ref{sec: invariant functions} and \ref{sec:independence}, we provide characterizations of finite ergodic components in terms of invariant functions and independence, respectively. In the last section, we study non-invariant probabilities.
\section{Preliminaries}\label{sec:prelimilaries}
\subsection{Capacity and Choquet integral}
Firstly, we recall a result in classical probability theory, which can be found in \cite[III. 7.2, Theorem 2 and Corollary 8]{DunfordSchwarz1988}.
\begin{lemma}[Vitali-Hahn-Saks]\label{lem:VHStheorem}
	Let $(\O, \mathcal{F})$ be a measurable space, and $\{P_n\}_{n\in\mathbb{N}}$ be a sequence  of probabilities. Suppose that  for each $A\in \mathcal{F}$, the  limit  $\lim _{n \rightarrow \infty} P_n(A)$ exists, denoted by $Q(A)$. Then $Q$ is a probability on $(\O,\mathcal{F})$. If we further suppose that each $P_n$ is absolutely continuous with respect to the probability $Q$, then the absolute continuity of  $P_n$ with respect to $Q$ is uniform in $n\in\mathbb{N}$, that is, for any $\epsilon>0$ there exists $\d>0$ such that for  any $A\in\mathcal{F}$, if  $Q(A)<\d$ then $ P_n(A)<\e$ for all $n\in\mathbb{N}$. 
\end{lemma}
Let $(\O,\mathcal{F})$ be an measurable space. We denote by $B_b(\O,\mathcal{F})$ the space of all $\mathcal{F}$-measurable and bounded functions. It is well-known that it is a Banach space with the maximum norm.
Given an upper probability, we call that a statement holds for $V$-almost surely ($V$-a.s. for short) if it holds on a set $A\in\mathcal{F}$ with $V(A^c)=0$.

Now we recall  Choquet integral introduced in \cite{Choquet1955} for an upper probability $V$ on $(\O,\mathcal{F})$,
\[\int_\O fdV=\int_{0}^{\infty}V(\{\o\in \O:f(\o)\ge t\})dt+\int_{-\infty}^{0}\left(V(\{\o\in \O:f(\o)\ge t\})-1\right)dt\]
for all $\mathcal{F}$-measurable functions $f$, where the integrals on the right-hand side are Lebesgue integrals. 
If there is no ambiguity, we will omit $\O$, and write $\int_\O$ as $\int$ for simplicity.

The following result provides a dominated convergence theorem in an upper probability space $(\O,\mathcal{F},V)$  with respect to Choquet integral, which was proved in \cite[Lemma 2.2]{FWZ2020}.
\begin{lemma}\label{lem:dominate convergence}
	For any $\mathcal{F}$-measurable functions $\{f_n\}_{n\in\mathbb{N}},$ $g$ and $h$ with $g\le f_n\le h$ for each $n\in\mathbb{N}$, and $\int gdV$, $\int hdV$ being finite, if  there exists a $\mathcal{F}$-measurable function $f$ such that  $\lim_{n\to\infty}f_n=f$, $V$-a.s. then 
	\[\lim_{n\to\infty}\int f_ndV=\int fdV.\]
\end{lemma}

\subsection{Invariant skeleton}
The following result provides the structure of the core of invariant upper probabilities, which was proved in Lemma 2.12 and Theorem 3.2 of \cite{FHLZ}. 
\begin{lemma}\label{lem:core of invariant lower probabilities}
	Let $(\O,\mathcal{F},V,T)$ be an UPPS. Then
	\begin{enumerate}
		\item[(i)]	given any $P\in \operatorname{core}(V)$, there exists a unique $\hat{P}\in \mathcal{M}(T)\cap\operatorname{core}(V)$ such that $P(A)=\hat{P}(A)$ for any $A\in\mathcal{I}$;
		\medskip
		\item[(ii)] given $A\in\mathcal{F}$, if for any $P\in\operatorname{core}(V)$, $\hat{P}(A)=0$, then $V(A)=0$.
	\end{enumerate}
	If we further assume that $V$ is ergodic, then $\mathcal{ M}(T)\cap \mathrm{core}(V)$ has only one element, denoted by $Q$, which is ergodic, and  $V\ll Q$, i.e., for any $A\in\mathcal{F}$,  $Q(A)=0$ implies that $V(A)=0$.
\end{lemma} 
\begin{rem}
	In \cite{FHLZ}, the probability $\hat{P}\in\mathcal{ M}(T)\cap \mathrm{core}(V)$ was named  the invariant skeleton of the probability $P\in\mathrm{core}(V)$. Moreover, if $V$ is ergodic, the unique element $Q$ in $\mathcal{ M}(T)\cap\operatorname{core}(V)$ (also in $\mathcal{ M}(T)\cap\operatorname{core}(V)$) is the invariant skeleton for all $P\in\operatorname{core}(V)$.
\end{rem}
 
Recall that the set of all finitely additive probabilities is denoted by $\D(\O,\mathcal{F})$. Throughout this paper, we equip $\D(\O,\mathcal{F})$ with the topology induced by the weak* topology. Specifically, a net $\{P_\alpha\}_{\alpha \in I}$ converges to $P$ in the weak* topology if and only if $P_\alpha(A)$ converges to $P(A)$ for every $A \in \mathcal{F}$.

Let $V$ be an upper probability. Then $\mathrm{core}(V)$ is compact in the weak* topology (see, for example, \cite[Proposition 4.2]{MM}). Since $V$ is continuous, it follows that $\mathrm{core}(V)$ is a compact subset of $\D^\sigma(\O,\mathcal{F})$, the set of all probabilities on $(\O,\mathcal{F})$. We now prove that the set $\mathrm{core}(V) \cap \mathcal{M}(T)$ is also compact.
\begin{lemma}\label{lem:compactness of m}
The set $\mathrm{core}(V)\cap \mathcal{ M}(T)$ is a compact subset of $\D^\sigma(\O,\mathcal{F})$ in the weak* topology.
\end{lemma}
\begin{proof}
Since $\mathrm{core}(V)$ is compact, we only need to prove that $\mathrm{core}(V)\cap \mathcal{ M}(T)$ is closed. Indeed, for any net $\{P_\a\}_{\a\in I}\subset \mathrm{core}(V)\cap \mathcal{ M}(T)$ and $P_\a$  converges to some  $P\in\D(\O,\mathcal{F})$. Using the fact that $\mathrm{core}(V)$ is compact again, we know that $P\in \mathrm{core}(V)\subset \D^\s(\O,\mathcal{F})$. 

Now it suffices to prove that $P$ is invariant. For any $A\in\mathcal{F}$, as $P_\a\in\mathcal{ M}(T)$, we have
\[P(A)=\lim_{\a\in I}P_\a(A)=\lim_{\a\in I}P_\a(T^{-1}A)=P(T^{-1}A).\]
Thus, $P$ is invariant and we finish the proof.
\end{proof}

\begin{rem}
Note that the  set $\mathcal{M}(T)$ may not be compact in the weak* topology. For example, we consider $\O$ is the set of   rational numbers, $\mathcal{F}$ is the Borel $\sigma$-algebra on $\O$ and $T$ is the identity map on $\O$. Then $\D^\sigma(\O,\mathcal{F})=\mathcal{M}(T)$. Now we choose a sequence $\{t_n\}_{n=1}^\infty$ of rational numbers such that it converges to $\sqrt{2}$. Then for any $n\in\mathbb{N}$, we have $P_n:=\d_{t_n}\in \mathcal{M}(T)$. Now we prove that the sequence $\{P_n\}_{n=1}^\infty$ has no convergence subnet, which implies that $\mathcal{M}(T)$ is not compact. 

By contradiction, we assume that  there exists a subnet $\{t_\a\}_{\a\in I}$ of $\{t_n\}_{n=1}^\infty$ such that $\{P_\a\}_{\a\in I}$ converges to a probability $P\in \mathcal{M}(T)$. We denote by $<_I$, the partial order on the directed set $I$.  
Now for any $x\in\O$, by the assumption that $\lim_{n\to\infty}t_n=\sqrt{2}$, there exists $\a_x>0$ such that for any $\a>_I \a_x$, $t_\a\neq x$, and hence $P(\{x\})=\lim_{\a\in I}P_\a(x)=0$. Consequently, $P(\O)=0$, as $\O$ is a countable set. This is a contradiction with the fact that $P$ is a probability, and so   $\mathcal{M}(T)$ is not compact in the weak* topology.

Meanwhile,  this example also shows that $\D^\sigma(\O,\mathcal{F})$ is not compact in the weak* topology.
\end{rem}

\subsection{Convex analysis}
Denote by
$\operatorname{ex}(A)$ the set of all extreme points of a convex set $A$, that is, the elements of $A$
that cannot be written as convex combinations of other elements.  The following result is classical  in convex analysis (see \cite[Theorem 8.14]{simon_2011} for example).
\begin{lemma}[Krein-Milman Theorem]\label{lem:Krein–Milman Theorem}
	Let $A$ be a compact convex subset of a locally convex vector space $X$. Then $A$ is the closed convex hull of $\operatorname{ex}(A)$.
\end{lemma}
\begin{rem}
	Since $\operatorname{core}(V)$ is a compact convex subset of the space consisting of all finitely additive measures under the weak* topology (see \cite[Proposition 4.2]{MM} for example), $\operatorname{core}(V)$  satisfies the condition in Lemma \ref{lem:Krein–Milman Theorem}.
\end{rem}

\section{Local ergodic decomposition of invariant upper probabilities}\label{sec:ergodic decomposition}
In this section, we will prove a local ergodic decomposition on $\operatorname{core}(V)$, which can be viewed as an extension of the result  in \cite{sheng2023continuous}, and moreover, we provide a Birkhoff's ergodic theorem for UPPSs with finite ergodic components.
Firstly, we have the following observation (see also \cite{sheng2023continuous}).
\begin{lemma}\label{lem: simple obser.}
	Let $(\O,\mathcal{F})$ be a measurable space, and $V$ be an upper probability. Then for any sequence of disjoint measurable sets $\{A_n\}_{n=1}^\infty$, \[\lim_{n\to\infty}V(A_n)=0.\]
\end{lemma} 
\begin{proof}
	Consider $\bar{A}_k=\cup_{n=k}^\infty A_n$ for each $k\in\mathbb{N}$. Then $\bar{A}_k\downarrow\emptyset$ as $k\to\infty$, which together with the continuity of $V$ implies that $\limsup_{n\to\infty}V(A_n)\le\lim_{k\to\infty}V(\bar{A}_k)=0$.
\end{proof}
As a corollary, we have the following.
\begin{lemma}\label{lem:finitely many probability}
	Let $(\O,\mathcal{F},V,T)$ be an UPPS. Then $\mathcal{ M}^e(T)\cap\operatorname{core}(V)$ has at most finitely many elements.
\end{lemma}
\begin{proof}
	By contradiction, assume that there exist infinitely many ergodic probabilities $\{Q_i\}_{i\in\mathbb{N}}\subset\operatorname{core}(V)$. Since $Q_i$ and $Q_j$ are singular to each other for any $i\neq j\in\mathbb{N}$, it follows that there exist $A_1,A_2\ldots\in\mathcal{F}$ such that for each $i\in\mathbb{N}$, $Q_i(A_i)=1$, and for each $i\neq j$, $A_i\cap A_j=\emptyset$,  and $Q_i(A_j)=0$. Thus, $V(A_i)=1$ for any $i\in\mathbb{N}$. This is a contradiction to Lemma \ref{lem: simple obser.}.
\end{proof}
\begin{rem}
	The set $\mathcal{ M}^e(T)\cap \operatorname{core}(V)$ may be empty. For example, $V$ is a probability which is invariant but not ergodic.
\end{rem}

Now we will study the structure of $\mathcal{ M}(T)\cap \operatorname{core}(V)$. 
It follows from $\mathcal{ M}^e(T)=\operatorname{ex}(\mathcal{ M}(T))$ (see \cite[Theorem 6.10]{Walters1982} for example)	that we have the following result.
\begin{lemma}\label{lem: extreme has no condition}
	Let $(\O,\mathcal{F},V,T)$ be an UPPS. Then 
	\[\operatorname{core}(V)\cap \mathcal{M}^e(T)\subset \operatorname{ex}(\operatorname{core}(V)\cap \mathcal{ M}(T)).\]
\end{lemma}

Now we prove the equivalence between the finite ergodic components and the condition that $V(\mathcal{I})=\{0,1\}$, and also provide characterizations of them from the convex analysis perspective. Let us begin with some lemmas.
\begin{lemma}\label{lem:control}
	Let $(\O,\mathcal{F},V,T)$ be an UPPS of finite ergodic components. Given any $\widetilde{P}\in\operatorname{core}(V)$, then for any  $Q\in\mathcal{ M}(T)$ with $Q\ll \widetilde{P}$, we have $Q\in\operatorname{core}(V)$.
\end{lemma}
\begin{proof}
	First, we prove that for any $A\in\mathcal{I}$, $Q(A)\le V(A)$.
	%	Since $V$ is of finite ergodic component,  there exist $n\ge1$, and $A_1,\ldots,A_n\in\mathcal{F}$ with  $\O=\cup_{i=1}^nA_i$,  $T^{-1}A_i=A_i$ for each $i$, and $A_i\cap A_j=\emptyset$, when $i\neq j$ such that  $V_i$  is an ergodic upper probability, where
	%	\[V_i(A):=\sup\{P(A):P\in\operatorname{core}(V)\text{ and } P(A_i)=1\}\text{ for any }A\in\mathcal{F}.\]
	By contradiction, if there exists $A\in\mathcal{I}$ such that $Q(A)>V(A)\ge V_i(A)\in\{0,1\}$, then $V(A)=0$, which implies that $\widetilde{P}(A)=0$, as $\widetilde{P}\in\operatorname{core}(V)$. Since $Q\ll\widetilde{P}$, it follows that $Q(A)=0$, which is a contradiction.
	
	Now we prove $Q\in\operatorname{core}(V)$. Define a functional on $B(\O,\mathcal{F})$ by 
	\[I(f)=\sup_{P\in\operatorname{core}(V)}\int fdP.\]
	By the linearity of integral  with respect to $P\in\operatorname{core}(V)$, it is easy to check that $I$ is a sublinear functional. Let $J(f):=\int fdQ$ for any $f\in B(\O,\mathcal{I})$. Since $Q(A)\le V(A)$ for any $A\in\mathcal{I}$, it follows that 
	\[J(f)\le I(f)\text{ for any }f\in B(\O,\mathcal{I}).\]
	By Hahn-Banach dominated extension theorem, there exists a linear functional $J^*:B(\O,\mathcal{F})\to\mathbb{R}$ such that 
	\begin{equation}\label{eq:2}
		J^*(f)=J(f) \text{ for any }f\in B(\O,\mathcal{I})
	\end{equation}
	and 
	\begin{equation}\label{eq:1}
		J^*(f)\le I(f)\text{ for any }f\in B(\O,\mathcal{F}).
	\end{equation}
	By Riesz representation theorem, there exists $P^*\in\D(\O,\mathcal{F})$ such that 
	\[J^*(f)=\int f dP^*\text{ for any }f\in B(\O,\mathcal{F}).\]
	By \eqref{eq:1}, one has that for any $A\in\mathcal{F}$, $P^*(A)=J^*(1_A)\le I(1_A)=V(A)$. Therefore, $P^*\in\operatorname{core}(V)$. In particular, it must be a probability, by the continuity of $V$. By \eqref{eq:2}, we have that $P^*|_{\mathcal{I}}=Q|_{\mathcal{I}}$. Since $P^*\in\operatorname{core}(V)$, by Lemma \ref{lem:core of invariant lower probabilities}, there exists a unique $\hat{P^*}\in\operatorname{core}(V)\cap \mathcal{ M}(T)$ such that $P^*|_{\mathcal{I}}=\hat{P^*}|_{\mathcal{I}}$. Thus, $Q|_{\mathcal{I}}=\hat{P^*}|_{\mathcal{I}}$. As $Q,\hat{P^*}\in\mathcal{ M}(T)$, by Lemma 2.8 in \cite{FHLZ}, it follows that $Q=\hat{P^*}\in\mathcal{ M}(T)\cap\operatorname{core}(V)$. The proof is completed.
\end{proof}

\begin{lemma}\label{lem:minimal invariant set}
	Let $(\O,\mathcal{F},V,T)$ be an UPPS. If $V(\mathcal{I})=\{0,1\}$, then for any $A\in\mathcal{I}$ with $V(A)=1$, there exists a subset $A'\in\mathcal{I}$ of $A$ with $V(A')=1$ such that for any $B\in \mathcal{I}$ with $B\subset A'$ and $V(B)=1$, we have $V(A'\setminus B)=0$.
\end{lemma}
\begin{proof}
	By contradiction, we assume that for any  subset $A'\in\mathcal{I}$ of $A$ with $V(A')=1$ there exists a  nonempty $B\in \mathcal{I}$ with $B\subset A'$ and $V(B)=1$, $V(A'\setminus B)>0$.  Hence $V(A'\setminus B)=1$, as $V(\mathcal{I})=\{0,1\}$. Denote this $A'\setminus B$ by $A_1$. Repeating this argument with $A'=B$, we obtain another set $B$ with $V(B)=1$ and $V(A'\setminus B)=1$. Denote this new $A'\setminus B$ by $A_2$. The process can continue forever. Thus, we can obtain a sequence of disjoint subsets $\{A_n\}_{n=1}^\infty$ with $V(A_n)=1$ for each $n\in\mathbb{N}$. This contradicts to Lemma \ref{lem: simple obser.}.
\end{proof}

Now we are able to provide some characterizations of finite ergodic components.
\begin{theorem}\label{thm:1}
	Let $(\O,\mathcal{F},V,T)$ be an UPPS.  Then the following four statements are equivalent:
	\begin{enumerate}
		\item [(i)] $V(\mathcal{I})=\{0,1\}$; 
		\item [(ii)] $V$ is of finite ergodic components;
		\item [(iii)] the set $\mathcal{ M}^e(T)\cap\operatorname{core}(V)=\{Q_1,\ldots,Q_n\}$, and  for any $P\in\mathcal{M}(T)\cap \operatorname{core}(V)$  there exist $0\le \a_1,\ldots,\a_{n}\le 1$ with $\sum_{i=1}^{n}\a_i=1$  such that 
		\[P=\a_1Q_1+\ldots+\a_nQ_n,\] where the decomposition is unique;
		\item[(iv)] $\operatorname{ex}(\operatorname{core}(V)\cap \mathcal{M}(T))=\operatorname{core}(V)\cap \mathcal{M}^e(T)\neq \emptyset$.
	\end{enumerate}
\end{theorem}
\begin{proof}
	(i) $\Rightarrow$ (ii).  Firstly, we use  Lemma \ref{lem:finitely many probability} and Lemma \ref{lem:minimal invariant set} to prove that there exist $n\in\mathbb{N}$ and $\{A_1,\ldots, A_n\}\subset \mathcal{I}$ satisfying  the following two properties:
	\begin{enumerate}
		\item [(1)]	$\O=\cup_{i=1}^nA_i, A_i\cap A_j=\emptyset, \text{ }i\neq j,\text{ and }V(A_i)=1,\text{ }i,j\in\{1,2,\ldots,n\};$
		\item [(2)] for any $i\in\{1,2,\ldots,n\}$ and $B\in\mathcal{I}$ with $B\subset A_i$ if $V(B)=1$, then $V(A_i\setminus B)=0$.
	\end{enumerate}
	Indeed, as $\O\in\mathcal{I}$ and $V(\O)=1$, by Lemma \ref{lem:minimal invariant set},  there exists $\tilde{A}_1\in\mathcal{I}$ with $V(\tilde{A}_1)=1$ and for any $B\in\mathcal{I}$ with $V(B)=1$ and $B\subset \tilde{A}_1$, then $V(\tilde{A}_1\setminus B)=0$. If $V(\O\setminus\tilde{A}_1)=0$,   we let $A_1=\O$. Then (1) and (2) above hold for $n=1$, and hence the construction is completed. Otherwise, we let $A_1=\tilde{A}_1$, and use the same argument on $\O\setminus A_1$.   Lemma \ref{lem: simple obser.} guarantees this process terminates at a finite number of steps, and we obtain a constant $n\in\mathbb{N}$ and a collection of sets $\{A_1,\ldots, A_n\}\subset \mathcal{I}$ satisfying (1) and (2).

	For each $i\in\{1,2,\ldots,n\}$, let
	$V_i$ be defined as in \eqref{eq:maindeifnition}.
	Now we prove $V_i$ is ergodic. Indeed, for any $A\in\mathcal{I}$, we have $A\cap A_i\in\mathcal{I}$. If $V(A\cap A_i)=0$, then $V_i(A)=V_i(A\cap A_i)\le 		V(A\cap A_i)=0$. If $V(A\cap A_i)=1$,  by (2), we have \[V(A_i\setminus(A\cap A_i))=0,\]
	which implies that $P(A^c)=0$, for any $P\in\operatorname{core}(V)$ with $P(A_i)=1$, and hence $V_i(A^c)=0$. Now we have proven that for any $A\in\mathcal{I}$, either $V_i(A)=0$ or $V_i(A^c)=0$. Thus, $V_i$ is ergodic.
	
	(ii) $\Rightarrow$ (iii). Since $V$ is of finite ergodic components,  there exist $n\in\mathbb{N}$, and $A_1,\ldots,A_n\in\mathcal{F}$ with  $\O=\cup_{i=1}^nA_i$, $T^{-1}A_i=A_i$ for each $i$ and $A_i\cap A_j=\emptyset$, when $i\neq j$, such that for each $i\in\{1,2,\ldots,n\}$, $V_i$ as defined in \eqref{eq:maindeifnition}  is an ergodic upper probability.
	For each $i\in\{1,2,\ldots,n\}$, since $V_i$ is ergodic, it follows from Lemma \ref{lem:core of invariant lower probabilities} that there exists a unique  $Q_i\in\operatorname{core}(V_i)\cap\mathcal{ M}^e(T|_{A_i})\subset\operatorname{core}(V)$. Now we prove $\mathcal{ M}^e(T)\cap\operatorname{core}(V)=\{Q_1,\ldots,Q_n\}$. It is obvious that $\{Q_1,\ldots,Q_n\}\subset\mathcal{ M}^e(T)\cap\operatorname{core}(V)$. Conversely, as any two different ergodic probabilities are singular,  if  there exists another ergodic probability $Q$ in $\operatorname{core}(V)$ then there exists $A\in\mathcal{I}$ with $Q(A)=1$ such that $A\cap A_i=\emptyset$ for all $i\in\{1,2,\ldots,n\}$. This is a contradiction with the assumption that $\O =\cup_{i=1}^nA_i$. Thus, 
	\[\mathcal{ M}^e(T)\cap\operatorname{core}(V)=\{Q_1,\ldots,Q_n\}.\]
	Given any $P\in\mathcal{M}(T)\cap \operatorname{core}(V)$, define
	\[P_i(A)=\frac{P(A\cap A_i)}{P(A_i)},\text{ for all }i\in\{1,2,\ldots,n\} \text{ such that }P(A_i)>0.\]
	Now we prove for any  fixed $i\in\{1,2,\ldots,n\}$,  $P_i\ll Q_i$. Indeed, for any $A\in\mathcal{I}$ with $Q_i(A)=0$. Since $V_i$ is ergodic and $Q_i\in \mathrm{core}(V_i)\cap \mathcal{M}^e(T)$, it follows from Lemma \ref{lem:core of invariant lower probabilities} that $V_i\ll Q_i$. Hence, one has 
	$V_i(A)=0$. Thus,  $P(A\cap A_i)=0$, which shows that $P_i(A)=0$ and hence $P_i\ll Q_i$. By Lemma \ref{lem:control}, we have that $P_i\in\operatorname{core}(V_i)\cap \mathcal{ M}(T)$, which together with the fact that $V_i$ is ergodic, implies that $Q_i=P_i$. Thus, 
	\[P=P(A_1)Q_1+\ldots+P(A_n)Q_n.\]
	The uniqueness of this decomposition is obvious.

	(iii) $\Rightarrow$ (iv).  This is a direct corollary of Lemma \ref{lem:Krein–Milman Theorem} and  Lemma \ref{lem: extreme has no condition}.
	
	(iv) $\Rightarrow$ (i). By Lemma \ref{lem:finitely many probability}, $\operatorname{core}(V)\cap \mathcal{M}^e(T)$ has at most finitely many elements, and we denote
	\[\operatorname{core}(V)\cap \mathcal{M}^e(T)=\{Q_1,\ldots,Q_n\}.\]
	By the assumption $\operatorname{ex}(\operatorname{core}(V)\cap \mathcal{M}(T))=\operatorname{core}(V)\cap \mathcal{M}^e(T)$, we know that for any $P\in \operatorname{core}(V)\cap \mathcal{ M}(T)$, there exists  $0\le \a_1,\ldots,\a_{n}\le 1$ with $\sum_{i=1}^{n}\a_i=1$ such that 
	\begin{equation}\label{eq:1234}
		P=\sum_{i=1}^n\a_iQ_i.
	\end{equation}
	Thus, for $A\in\mathcal{I}$, if there exists $i\in\{1,2,\ldots,n\}$ such that $Q_i(A)=1$, then $V(A)=1$. Otherwise, for any $P\in \operatorname{core}(V)\cap \mathcal{ M}(T)$, from \eqref{eq:1234}, one has  $P(A)=0$, which together with Lemma \ref{lem:core of invariant lower probabilities}, implies that $V(A)=0$. Thus, $V(\mathcal{I})=\{0,1\}$.
\end{proof}
\begin{rem}
	(i) $\Rightarrow$ (iii) was proved by a different method in \cite{sheng2023continuous}.
\end{rem}
As a corollary, we provide a type of Birkhoff's ergodic theorem for UPPSs with finite ergodic components.
\begin{theorem}\label{thm:Birkhoff's ergodic}
	Let $(\O,\mathcal{F},V,T)$ be	an UPPS. Then the following three statements are equivalent:
	\begin{enumerate}
		\item [(i)]  $V$ is of finite ergodic components;
		\medskip
		\item [(ii)]there exist  $Q_1,\ldots,Q_n\in\mathcal{ M}^e(T)\cap\operatorname{core}(V)$ and $A_1,\ldots,A_n\in\mathcal{I}$ satisfying that  $\O=\cup_{j=1}^nA_j$, $Q_j(A_j)=1$, and $A_j\cap A_k=\emptyset$ when $j\neq k$ such that for any $f\in B(\O,\mathcal{F})$,  
		\begin{equation}\label{eq:4.81}
			\lim_{N\to\infty}\frac{1}{N}\sum_{i=0}^{N-1}f(T^i\o)=\sum_{j=1}^{n}\left(\int fdQ_j\right) \cdot1_{A_j}(\o),\text{ for }V\text{-a.s. }\o\in\O.
		\end{equation}
		\medskip
		\item [(iii)]there exist  $Q_1,\ldots,Q_n\in\mathcal{ M}(T)\cap\operatorname{core}(V)$ and $A_1,\ldots,A_n\in\mathcal{I}$ satisfying that  $\O=\cup_{j=1}^nA_j$, $Q_j(A_j)=1$, and $A_j\cap A_k=\emptyset$ when $j\neq k$ such that for any $f\in B(\O,\mathcal{F})$,  
		\begin{equation}\label{eq:4.8}
			\lim_{N\to\infty}\frac{1}{N}\sum_{i=0}^{N-1}f(T^i\o)=\sum_{j=1}^{n}\left(\int fdQ_j\right) \cdot1_{A_j}(\o),\text{ for }V\text{-a.s. }\o\in\O.
		\end{equation}
	\end{enumerate}  
	%Furthermore, 	(i) and (ii) are equivalent such that  $\mathcal{ M}(T)\cap\operatorname{core}(V)=\mathcal{ M}^e(T)\cap\operatorname{core}(V)=\{Q_1,\ldots,Q_n\}$.
\end{theorem}
\begin{proof}
	(i) $\Rightarrow$ (ii).	Let $Q_j$, $j=1,2,\ldots,n$ be as in Theorem \ref{thm:1}. Applying the law of large numbers \cite[Theorem 3.3]{FHLZ} for the ergodic UPPS $(A_j,\mathcal{F}_j,V_j,T)$, there exists  $A_j'\subset A_j$ with $V_j(A_j\setminus A_j')=0$.
	\[\lim_{N\to\infty}\frac{1}{N}\sum_{i=0}^{N-1}f(T^i\o)=\int fdQ_j ,\text{ for any }\o\in A_j'.\]
	Let $A'=\cup_{j=1}^n A'_j$. Then for each  $j\in\{1,2,\ldots,n\}$, $Q_j(A'^c)=0$, which together with Lemma \ref{lem:core of invariant lower probabilities}, implies that $V(A'^c)=0$. Also it is obvious that $1_{A_j}=1_{A_j'}$, $V$-a.s. So \eqref{eq:4.81} holds and the proof is completed.
	
	(ii) $\Rightarrow$ (iii) is trivial.
	
	(iii) $\Rightarrow$ (i).	If \eqref{eq:4.8} holds, we show  $Q_j$ is ergodic, for each $j=1,2,\ldots,n$. This can be seen from the following. Given any $A\in\mathcal{I}$, considering $f=1_A$ in \eqref{eq:4.8}, we have that 
	\begin{equation}\label{10.43}
		1_A=\sum_{j=1}^nQ_j(A)1_{A_j},\text{ }V\text{-a.s.}
	\end{equation}
	This implies that for any $j\in\{1,2,\ldots,n\}$, $Q_j(A)\in\{0,1\}$. Thus, $Q_j$ is ergodic for each $j\in\{1,2,\ldots,n\}$.
	
	Now we prove that for any fixed $j\in\{1,2,\ldots,n\}$, $V_j$ is ergodic. We only need to prove, for any $A\in\mathcal{I}$, if $Q_j(A)=0$, then $V_j(A)=0$. Indeed, given any such $A$, by \eqref{10.43}, we have that 
	\[V(\{\o\in\O:1_A(\o)=0\}^c\cap A_j)=0,\]
	which implies that $V_j(A)\le V(A\cap A_j)=0$. Thus, $V_j$ is ergodic for each $j\in\{1,2,\ldots,n\}$, and then $V$ is of finite ergodic components.
\end{proof}
\begin{rem}
	In \cite{CMM2016} and \cite{sheng2023continuous}, the authors proved if $V(\mathcal{I})=\{0,1\}$, then
	\[\inf_{P\in\operatorname{core}(V)}\int fdP\le \lim_{N\to\infty}\frac{1}{N}\sum_{i=0}^{N-1}f(T^i\o)\le \sup_{P\in\operatorname{core}(V)}\int fdP,\text{ for }V\text{-a.s. }\o\in\O.\]
	However, our result provides  concrete values for this Birkhoff's limit, which are multiple values, in contrast to the result in the ergodic regime \cite{FHLZ} and Birkhoff's ergodic theorem in the classical probability case. Moreover, our result also gives the equivalence of $V(\mathcal{I})=\{0,1\}$, and the law of large numbers. This is impossible without knowing the precise (multi-) values of the Birkhoff's limit.
\end{rem}

Finally, we provide a decomposition for all invariant upper probabilities without assuming the UPPS being of finite ergodic components. Let us begin with a lemma, which was proved in the proof of \cite[Theorem 6.10 (iv)]{Walters1982}.
\begin{lemma}\label{lem:Leb deco}
	Let $(\O,\mathcal{F})$ be a measurable space, and $T:\O\to\O$ be a measurable invertible transformation. For any $P,Q\in\mathcal{ M}(T)$, there exist  $k,l>0$ with $k+l=1$ and  $P_a,P_s\in \mathcal{ M}(T)$  such that
	\[P=kP_a+lP_s, \\ P_a\ll Q\text{ and }P_s\perp Q.\]
\end{lemma}

%\begin{proof}
%	Recall the proof of Lebesgue decomposition theorem by von Neumann. Let $\nu=P+Q$ and define a bounded linear functional on $L^2(\O,\mathcal{F},\nu)$ by 
%	\[f\mapsto\int fdP.\]
%	By Riesz representation theorem, there exists $g\in L^2(\O,\mathcal{F},\nu)$ such that 
%	\[\int fdP=\int fgd\nu\text{ for any }f\in L^2(\O,\mathcal{F},\nu).\]
%	Let $A=\{\o\in\O:0\le g(\o)<1\}$ and $S=\{\o\in\O: g(\o)=1\}$, and define
%	\[P_a=\frac{P(\cdot\cap A)}{P(A)}\text{ and }P_s=\frac{P(\cdot\cap S)}{P(S)}.\]
%	Then 
%	\[P=kP_a+lP_s, \\ P_a\ll Q\text{ and }P_s\perp Q,\]
%	where $k=P(A)$ and $l=P(S)$. As $P$, $P_a$ and $P_s$ are probabilities, one has %$k+l=1$.
	
%	Thus, by the uniqueness of this decomposition, we only need to prove $g$ is $T$-invariant, $P$-a.s., which implies that  $P_a,P_s\in \mathcal{ M}(T)$. Indeed, as $\nu,P$ are invariant, and $T$ is invertible, it follows that for any $f\in L^2(\O,\mathcal{F},\nu)$,
%	\[\int fdP=\int f\circ T^{-1}dP=\int (f\circ T^{-1})gd\nu=\int f(g\circ %T)d\nu.\]
%	By the uniqueness of Riesz representation theorem, we have that $g\circ T=g$, $\nu$-a.s. and hence $P$-a.s. It follows that $A,S\in\mathcal{I}$, and so $P_a,P_s\in\mathcal{ M}(T)$.
%\end{proof}
The following result provides a decomposition of any invariant probability within the core of any invariant upper probability, without assuming finitely many ergodic components.
\begin{theorem}\label{thm:decompostion for all invariant upps}
	Let $(\O,\mathcal{F},V,T)$ be an UPPS, where $T$ is invertible.   Suppose that  $\mathcal{ M}^e(T)\cap\operatorname{core}(V)=\{Q_1,\ldots,Q_n\}$. Then for any $P\in\mathcal{ M}(T)\cap\operatorname{core}(V)$, there exist $0\le \a_1,\ldots,\a_{n+1}\le 1$ with $\sum_{i=1}^{n+1}\a_i=1$ and $Q_{n+1}\in \mathcal{ M}(T)\cap\operatorname{core}(V)$, which is singular with respect to all $Q_i$, $i=1,2,\ldots,n$ such that 
	\[P=\a_1Q_1+\ldots+\a_nQ_n+\a_{n+1}Q_{n+1}.\]
	Moreover, the decomposition is unique.
\end{theorem}
\begin{proof}
	Let 
	\[V'(A)=\sup\{P(A):P\in\operatorname{core}(V)\text{ and }P\ll \frac{1}{n}\sum_{i=1}^nQ_i\}\text{ for any }A\in\mathcal{F}.\]
	It is easy to see that  $V'$ is an upper probability. Now we prove it is $T$-invariant. Indeed, for any $P\in\operatorname{core}(V)$ with $P\ll \frac{1}{n}\sum_{i=1}^nQ_i$, one has $P\circ T, P\circ T^{-1}\in\operatorname{core}(V)$ and $P\circ T, P\circ T^{-1}\ll \frac{1}{n}\sum_{i=1}^nQ_i$, as $Q_i$, $i=1,2,\ldots,n$ and $V$ are $T$-invariant, and $T$ is invertible. Thus, for any $A\in\mathcal{F}$ and $P\in\operatorname{core}(V)$ with $P\ll \frac{1}{n}\sum_{i=1}^nQ_i$, 
	$P(T^{-1}A)\le V'(A),$
	which implies that 
	\[V'(T^{-1}A)\le V'(A).\]
	Conversely, note that for any $A\in\mathcal{F}$ and $P\in\operatorname{core}(V)$ with $P\ll \frac{1}{n}\sum_{i=1}^nQ_i$,
	\[P(A)=P\circ T(T^{-1}A)\le V'(T^{-1}A),\]
	which implies that
	\[V'(A)\le V'(T^{-1}A).\]
	Thus, $V'$ is $T$-invariant.  Since   $\operatorname{core}(V)\cap \mathcal{ M}^e(T)=\{Q_1,\ldots,Q_n\}$, it follows from the construction that $\operatorname{core}(V')\cap \mathcal{ M}^e(T)=\{Q_1,\ldots,Q_n\}$.
	
	Now we prove that  $V'(\mathcal{I})=\{0,1\}$. 
	Indeed, for any $A\in\mathcal{I}$, if there exists $i\in\{1,2,\ldots,n\}$ such that $Q_i(A)>0$, then $Q_i(A)=1$, as $Q_i$ is ergodic. Thus, $V(A)=1$, as $V(A)\ge Q_i(A)$. Otherwise, for each $i\in\{1,2,\ldots,n\}$, $Q_i(A)=0$, which implies that for any $P\in\operatorname{core}(V)\text{ and }P\ll \frac{1}{n}\sum_{i=1}^nQ_i$, $P(A)=0$. Thus, $V'(A)=0$. As $A\in\mathcal{I}$ is arbitrary, we have $V'(\mathcal{I})=\{0,1\}$.
	
	Given any $P\in\mathcal{ M}(T)\cap \operatorname{core}(V)$, applying Lemma \ref{lem:Leb deco} on $P$ and $\frac{1}{n}\sum_{i=1}^nQ_i$, there exist  $k,l>0$ with $k+l=1$ and  $P_a,P_s\in \mathcal{ M}(T)$  such that
	\[P=kP_a+lP_s, \\ P_a\ll \frac{1}{n}\sum_{i=1}^nQ_i\text{ and }P_s\perp \frac{1}{n}\sum_{i=1}^nQ_i.\]
	Thus, by Lemma \ref{lem:control}, we have that $P_a\in\operatorname{core}(V')$, which together with the fact $V'(\mathcal{I})=\{0,1\}$, by Theorem \ref{thm:1}, there exist $0\le \ell^a_1,\ldots,\ell^a_{n}\le 1$ with $\sum_{i=1}^{n}\ell^a_i=1$  such that 
	\[P_a=\ell^a_1Q_1+\ldots+\ell^a_nQ_n.\] 
	Let $Q_{n+1}=P_s$, $\ell_i=\ell_i^a\cdot k$, and $\ell_{n+1}=l$. Then $Q_{n+1}\in\mathcal{ M}(T)$, and
	\[P=\ell_1Q_1+\ldots+\ell_nQ_n+\ell_{n+1}Q_{n+1}.\]

	If $\ell_{n+1}=0$, the proof is completed. Now, we check that $Q_{n+1}\in \mathcal{ M}(T)\cap\operatorname{core}(V)$ when  $\ell_{n+1}\neq 0$. Indeed, it is clear that $\mathcal{ M}(T)\cap\operatorname{core}(V)$ is convex, which together with Lemma \ref{lem:compactness of m}, implies it is compact convex subset in the weak* topology. Since $P\in  \mathcal{ M}(T)\cap\operatorname{core}(V)$, it follows from Lemma \ref{lem:Krein–Milman Theorem} that  there exists a net $\{\mu_\a\}_{\a \in I}\subset \mathcal{ M}(T)\cap\operatorname{core}(V)$, where for each $\a\in I$, $\mu_\a$ is the convex combination of elements in $\operatorname{ex}( \mathcal{ M}(T)\cap\operatorname{core}(V))$, such that \begin{equation}\label{eq:6.16}
		\lim_{\a\in I}\mu_\a=P .
	\end{equation}By Lemma \ref{lem: extreme has no condition},  $Q_1,\ldots,Q_n\in \operatorname{ex}( \mathcal{ M}(T)\cap\operatorname{core}(V))$, and so  for each $\a\in I$, $\mu_\a$ can be written as the following form
	\[\mu_\a=\sum_{i=1}^{n}\l_i^\a Q_i+\sum_{i=n+1}^{l_\a}\l_i^\a\nu_i,\ \text{ where } \sum_{i=1}^{l_\a}\l_i^\a=1,\ \nu_i\in \operatorname{ex}( \mathcal{ M}(T)\cap\operatorname{core}(V))\setminus\{Q_1,\ldots,Q_n\},\ \a\in I.\]
	This combined with \eqref{eq:6.16} and the uniqueness of the decomposition of $P$, we have that for each $i\in\{1,2,\ldots,n\}$,
	\[\lim_{\a\in I}\l_i^\a=\ell_i,\]
	and hence
	\[\lim_{\a\in I}\sum_{i=n+1}^{l_\a}\l_i^\a\nu_i= \lim_{\a\in I}\left(\mu_\a-\sum_{i=1}^{n}\l_i^\a Q_i\right)=P-\sum_{i=1}^{n}\ell_iQ_i=\ell_{n+1}Q_{n+1}.\]
	By the closedness of $ \mathcal{ M}(T)\cap\operatorname{core}(V)$, we have $Q_{n+1}\in  \mathcal{ M}(T)\cap\operatorname{core}(V)$. Now we finish the proof.
\end{proof}

\section{Finite ergodic components in terms of multiplicity of the eigenvalue of Koopman operator}\label{sec: invariant functions}
In this section, we provide the following characterization of finite ergodic components in terms of the space of invariant functions. 
Let $(\O,\mathcal{F},V,T)$ be  an UPPS.
Define  the Koopman operator $U_T:B_b(\O,\mathcal{F})\to B_b(\O,\mathcal{F})$ by $U_Tf=f\circ T$.
\begin{theorem}\label{thm:ergodicity of upper v.s. invariant functions}
	Let $(\O,\mathcal{F},V,T)$ be an UPPS. Then the following three statements are equivalent:
	\begin{enumerate}
		\item [(i)] $V$ is of finite ergodic components;
		\medskip
		\item [(ii)]there exists a measurable partition $\{A_1,\ldots,A_n\}$ of $\O$ with  $V(A_i)=1$, $i=1,2,\ldots,n$ such that for any $T$-invariant function $f$, i.e., $f\circ T=f$, there exists $c_1,\ldots,c_n\in\mathbb{C}$ such that 
		\[f=\sum_{j=1}^nc_j1_{A_j},\text{ }V\text{-a.s.};\]
		
		\item [(iii)] there exists a measurable partition $\{A_1,\ldots,A_n\}$ of $\O$ with  $V(A_i)=1$, $i=1,2,\ldots,n$ such that the eigenspace of the Koopman operator $U_T$ corresponding to eigenvalue $1$ is spanned by $\{1_{A_i}:i=1,2,\ldots,n\}$ and  the eigenvalue $1$ of $U_T$ is of multiplicity $n$.
	\end{enumerate}
\end{theorem}
\begin{proof}

	(i) $\Rightarrow$ (ii) is a direct result of Theorem \ref{thm:Birkhoff's ergodic}.
	
	(ii) $\Rightarrow$ (i). For any $A\in\mathcal{I}$, it is easy to see that $1_A$ is an $T$-invariant function, so by assumption (ii), there exists $c_1,\ldots,c_n\in\{0,1\}$ such that
	\[1_A=\sum_{j=1}^nc_j1_{A_j},\text{ }V\text{-a.s.}\]
	If $c_j=0$ for each $j\in\{1,2,\ldots,n\}$, then $1_A=0$, $V$-a.s., that is, $V(A)=0$; if there exists $c_j=1$, then $V(A)\ge V(A_j)=1$. Thus, $V(\mathcal{I})=\{0,1\}$, and hence it is of finite ergodic components.
	
	(ii) $\Leftrightarrow$ (iii) follows from the equivalence of $f\circ T=f$, $V$-a.s., i.e., $f$ is an $T$-invariant functions and $U_Tf=f$,  , $V$-a.s., i.e., $f$ is the eigenfunction of $U_T$ corresponding to the eigenvalue $1$. 
\end{proof}

\section{Finite ergodic components in terms of independence}\label{sec:independence}
In this section, we provide the following characterization of finite ergodic components in terms of ``weak'' independence.
\begin{theorem}\label{thm:ergodicity of upper}
	Let $(\O,\mathcal{F},V,T)$ be an UPPS. Then the following two statements are equivalent:
	\begin{enumerate}
		\item [(i)]$V$ is of finite ergodic components;
		\medskip
		\item [(ii)] there exist  $Q_1,\ldots,Q_n\in\mathcal{ M}(T)\cap\operatorname{core}(V)$ and $A_1,\ldots,A_n\in\mathcal{I}$ satisfying that  $\O=\cup_{j=1}^nA_j$, $Q_j(A_j)=1$, and $A_j\cap A_k=\emptyset$ when $j\neq k$ such that for any $B,C\in \mathcal{F}$,
		\begin{equation}\label{eq:123}
			\lim_{N\to\infty}\int\frac{1}{N}\sum_{i=0}^{N-1} 1_B\cdot(1_C\circ T^i)dV=\sum_{j=1}^nQ_j(C)(V(\cup_{k=0}^jA_k\cap B)-V(\cup_{k=0}^{j-1}A_k\cap B)),
		\end{equation}
		where $A_0=\emptyset$.
	\end{enumerate}
\end{theorem}
\begin{proof}
	(i) $\Rightarrow$ (ii). Suppose that $V$ is of finite ergodic components with respect to $Q_1,\ldots,Q_n$ and $A_1,\ldots,A_n$. 	 Computing by the definition of Choquet integral (see for example \cite[Proposition 4.10]{MM}), we obtain that 
	\begin{equation}\label{eq:10.2916.10}
	\int 1_B\left(\sum_{j=1}^nQ_j(C)1_{A_j}\right)dV=\sum_{j=1}^nQ_j(C)(V(\cup_{k=0}^jA_k\cap B)-V(\cup_{k=0}^{j-1}A_k\cap B)).
	\end{equation}

	 Given $B,C\in\mathcal{F}$,
	it follows from Lemma \ref{lem:dominate convergence} and  Theorem \ref{thm:Birkhoff's ergodic} that 
	\begin{align*}
		\lim_{N\to\infty}\int \frac{1}{N}\sum_{i=0}^{N-1}1_B(\o)1_C(T^i\o)dV(\o)&=\int 1_B\left(\sum_{j=1}^nQ_j(C)1_{A_j}\right)dV,
	\end{align*}
	which together with \eqref{eq:10.2916.10} implies (ii) holds.

	(ii) $\Rightarrow$ (i). By Theorem \ref{thm:1}, we only need to prove $V(\mathcal{I})=\{0,1\}$. Firstly, we prove for each fixed $j\in\{1,2,\ldots,n\}$, $Q_j$ is ergodic. Indeed,  for any $A\in\mathcal{I}$, let $B=A_j\setminus(A\cap A_j)$ and $C=A\cap A_j$ in \eqref{eq:123}.  Now one has  that 
		\begin{align*}
		0&=	\lim_{N\to\infty}\int\frac{1}{N}\sum_{i=0}^{N-1} 1_B\cdot(1_C\circ T^i)dV=
		\sum_{j=1}^nQ_j(C)(V(\cup_{k=0}^jA_k\cap B)-V(\cup_{k=0}^{j-1}A_k\cap B)),
		\end{align*}
which, together with \eqref{eq:10.2916.10}, demonstrates that
	\[0=	\int 1_B\left(\sum_{k=1}^nQ_k(C)1_{A_k}\right)dV=Q_j(A\cap A_j)V(A_j\setminus(A\cap A_j))=Q_j(A)V(A_j\setminus A).\]
	If $Q_j(A)>0$, then $V(A_j\setminus A)=0$, which implies that $Q_j(A^c)=Q_j(A_j\setminus A)=0$. Thus, $Q_j$ is ergodic.
	
	Now we prove $V(\mathcal{I})=\{0,1\}$. For any $A\in\mathcal{I}$, let $B=A^c$, and $C=A$ in \eqref{eq:123}. Then 
	\[\sum_{j=1}^nQ_j(A)(V(\cup_{k=0}^jA_k\cap A^c)-V(\cup_{k=0}^{j-1}A_k\cap A^c))=0.\]
%	Since $V(\cup_{k=0}^jA_k\cap A^c)-V(\cup_{k=0}^{j-1}A_k\cap A^c)\ge0$, we have that $Q_j(A)(V(\cup_{k=0}^jA_k\cap A^c)-V(\cup_{k=0}^{k-1}A_k\cap A^c))=0$ for all $j=1,2,\ldots,n$. 
Now if $Q_j(A)=0$ for all  $j=1,2,\ldots,n$, then by considering $B=\O$ and $C=A$ in \eqref{eq:123}, we have that $V(A)=0$. If there exists $j\in\{1,2,\ldots,n\}$ such that $Q_j(A)>0$, then by the ergodicity of $Q_j$, we have $Q_j(A)=1$, and hence $V(A)=1$. This shows that $V$ is of finite ergodic components.
	
\end{proof}
\begin{rem}
	When $n=1$, the right hand side of \eqref{eq:123} is equal to $Q_1(C)V(B)$, which was proved in Theorem 4.4 of \cite{FHLZ}. In that case, $V$ is ergodic.
\end{rem}

Moreover, we provide a characterization of ergodicity of  an upper probability via the asymptotic independence of probabilities in its core.
\begin{theorem}\label{thm:ergodic via independece of core}
	Let $(\O,\mathcal{F},V,T)$ be an UPPS. Then the following statements are equivalent:
	\begin{enumerate}
		\item [(i)] $V$ is of finite ergodic components;
		\medskip
		\item [(ii)]there exist  $Q_1,\ldots,Q_n\in\mathcal{ M}(T)\cap\operatorname{core}(V)$ and $A_1,\ldots,A_n\in\mathcal{I}$ satisfying that  $\O=\cup_{j=1}^nA_j$, $Q_j(A_j)=1$, and $A_j\cap A_k=\emptyset$ when $j\neq k$ such that  for any  $P\in\operatorname{core}(V)$, 
		\begin{equation}\label{eq:111111}
			\lim_{N\to\infty}\frac{1}{N}\sum_{i=0}^{N-1}P(B\cap T^{-i}C)=\sum_{j=1}^nQ_j(C)P(A_j\cap B)\text{ for any }B,C\in\mathcal{F}.
		\end{equation}
	\end{enumerate}
\end{theorem}
\begin{proof}
	(i) $\Rightarrow$ (ii). Denote $Q_1,\ldots,Q_n$ and $A_1,\ldots,A_n$ as in Theorem \ref{thm:1}. By Theorem \ref{thm:Birkhoff's ergodic} and dominated convergence theorem, one has that 
	\begin{align*}
		\lim_{N\to\infty}\frac{1}{N}\sum_{i=0}^{N-1}P(B\cap T^{-i}C)=&	\lim_{N\to\infty}\frac{1}{N}\sum_{i=0}^{N-1}\int 1_B(\o)\cdot 1_C(T^i\o) dP(\o)\\
		=&\int 1_B(\o)\cdot \sum_{j=1}^n Q_j(C)1_{A_j}(\o) dP(\o)\\
		=& \sum_{j=1}^n Q_j(C)P(A_j\cap B).
	\end{align*}
	
	(ii) $\Rightarrow$ (i). The proof can be obtained by a proof  similar to that of 	(ii) $\Rightarrow$ (i) in Theorem \ref{thm:ergodicity of upper}.
\end{proof}

\section{Application to non-invariant probabilities}

Instead of starting with  a given ergodic upper probability, let us  consider a probability space $(\O,\mathcal{F},P)$ and an invertible measurable transformation $T:\O\to\O$ (note that $P$ may be not $T$-invariant).

Under the	assumption that the limit 
\begin{equation}\label{eq:20.19}
	\lim_{N\to\infty}\frac{1}{N}\sum_{i=0}^{N-1}P(T^{-i}A)
\end{equation} exists for each $A\in\mathcal{F}$ and  $P(\mathcal{I})=\{0,1\}$, the authors  \cite{FHLZ} investigated the law of large numbers and independence for non-invariant probability $P$. With the help of our results for nonlinear probabilities, in this paper, we study a more general class of non-invariant probabilities. Denote by $\#(F)$ 
the cardinality of the set $F$. For convenience, we denote the following assumption:

\medskip

\noindent \textbf{Assumption $(*)$}:  $\#(P(\mathcal{I}))<\infty$. 

Now we prove that Assumption $(*)$ holds if and only if $\O$ has a ``good partition'' in the following sense.		
\begin{definition}\label{def1}
	A measurable partition $\a$ of $\O$ is said to be a finite irreducible invariant set partition (with respect to $P$ and $T$) if $\#(\a)<\infty$ and  for any $A\in \a$, 
	\begin{enumerate}
		\item [(i)]$P(A)>0$;
		\item [(ii)]$A\in\mathcal{I}$;
		\item [(iii)] none of $T$-invariant sets $B\subset A$ with $0<P(B)<P(A)$. 
	\end{enumerate}
In this case,	we call  $(\O,\mathcal{F},P)$   satisfies the finite irreducible invariant set partition assumption.
\end{definition}
\begin{rem}
	If  $(\O,\mathcal{F},P,T)$ has a  finite  irreducible invariant set partition $\a=\{A_1,\ldots,A_n\}$, then the partition is unique in the following sense: if there exists another finite irreducible invariant set partition $\a'=\{A_1',\ldots,A_m'\}$, then $m\ge n$ and there exists $j_1,\ldots,j_n\in\{1,2\ldots,m\}$ such that  for any $i=1,\ldots, n$ $P(A_i\D A_{j_i}')=0$.
\end{rem}
\begin{lemma}\label{lem:good partition}
	Let $(\O,\mathcal{F},P)$ be a probability space and $T:\O\to\O$ be a measurable transformation. Then Assumption $(*)$ holds if and only if  $(\O,\mathcal{F},P,T)$ satisfies the finite irreducible invariant set partition assumption. 
\end{lemma}
\begin{proof}
	$(\Rightarrow)$ We claim that for any $A'\in\mathcal{I}$ with $P(A')>0$, there exists $A\in\mathcal{I}$ with $A\subset A'$ satisfies (i)-(iii) in Definition \ref{def1}. Indeed, by contradiction, we assume that for any  subset $A\in\mathcal{I}$ of $A'$ with $P(A)>0$ there exists a  nonempty $B\in \mathcal{I}$ with $B\subset A$  such that $0<P(B)<P(A)$. Inductively, we can obtain a sequence $\{B_n\}_{n=1}^\infty$ consisting of invariant subsets of $A'$ such that $P(B_n)>P(B_{n+1})>0$ for each $n\in\mathbb{N}$. This contradicts with Assumption $(*)$, and hence the claim holds.
	
	We now begin with the set $\O\in\mathcal{I}$. Then by the result above there exists $A_1\in\mathcal{I}$ satisfying (i)-(iii) in Definition \ref{def1}.  Next, let $\O_1=\O\setminus A_1$. Then $\O_1\in\mathcal{I}$, and hence using the claim again, we obtain $A_2\in\mathcal{I}$ satisfying (i)-(iii) in Definition \ref{def1}. According to Assumption $(*)$, this process will conclude in a finite number of steps. Thus, we obtain a finite irreducible invariant set partition, that is, $(\O,\mathcal{F},P,T)$ satisfies the finite irreducible invariant set partition assumption. 
	
	($\Leftarrow$) Let $\a=\{A_1,\ldots,A_n\}$ be a finite irreducible invariant set partition. Then for any $A\in\mathcal{I}$ and $j\in\{1,2,\ldots,n\}$, we have  $P(A\cap A_j)=P(A_j)$ or $0$. Let $\mathcal{K}(A):=\{A_j:P(A\cap A_j)\neq 0\}$. Thus, for any $A\in\mathcal{I}$, 
	\begin{align*}
		P(A)=	\left\{ 
		\begin{array}{ll}
			\sum_{\{A_j:P(A\cap A_j)\neq 0\}}P(A_j), & \text{ if }\mathcal{K}(A)\neq \emptyset \\
			0, & \text{ otherwise.}
		\end{array}
		\right.
	\end{align*}
	Thus, we have 
	\[\#(P(\mathcal{I}))\le 2^n<\infty, \]
	so the proof is completed.
\end{proof}

For convenience, if $(\O,\mathcal{F},P,T)$ has a  finite  irreducible invariant set partition $\a$, we always assume this partition is of the form $\a=\{A_1,\ldots,A_n\}$.		
Under the finite	invariant set partition assumption, define the $\s$-algebra $\mathcal{F}_j:=\{F\cap A_j:F\in\mathcal{F}\}$ and $P_j(A)=P(A\cap A_j)/P(A_j)$ on $\mathcal{F}_j$ for $j=1,2,\ldots,n$. Then  each $j\in\{1,2,\ldots,n\}$, $(A_j,\mathcal{F}_j,P_j)$ is a probability space.
If we suppose that the limit  \eqref{eq:20.19} exists for  each $A\in\mathcal{F}$, then the limit
\begin{equation}\label{eq:5.1b}
	\lim_{N\to\infty}\frac{1}{N}\sum_{i=0}^{N-1}P_j(T^{-i}A)
\end{equation} exists for each $A\in\mathcal{F}
$, denoted by $Q_j(A)$.  Then by 
Vitali-Hahn-Saks's theorem (see Lemma \ref{lem:VHStheorem}), we have $Q_j$ is a probability.  It is obvious that $Q_j\in\mathcal{ M}(T)$ and $P_j|_{\mathcal{I}}=Q_j|_{\mathcal{I}}$. 

Now under the above assumption, we construct an upper probability $V_j$ on $(A_j,\mathcal{F}_j)$ for each $j=1,2,\ldots,n$. For fixed  $j=1,2,\ldots,n$, let us begin with a claim that  $P_j\circ T^{-i}$ is absolutely continuous with respect to $Q_j$ for each $i\in\mathbb{Z}$. Indeed, for any $A\in\mathcal{F}_j$ with $Q_j(A)=0$, then $Q_j(A^\infty)=0$, where $A^\infty=\cup_{i=-\infty}^\infty T^{-i}A$. As $A^\infty\in\mathcal{I}$, it follows that for any $i\in\mathbb{Z}$,
\[P_j(T^{-i}A)\le P_j(T^{-i}A^\infty)=P_j(A^\infty)=Q_j(A^\infty)=0,\]
proving this claim. Define 
\[\l^j_{M,N}=\frac{1}{M+N+1}\sum_{i=-M}^NP_j\circ T^{-i},\text{ }M,N\ge0\] and 
\begin{equation}\label{eq:8=}
	V_j(A)=\sup_{M,N\in\mathbb{Z}_+}\{\l^j_{M,N}(A)\}\text{ for each }A\in\mathcal{F}_j,
\end{equation}
where $\mathbb{Z}_+$ is the set of all non-negative integrals.
Now using Vitali-Hahn-Saks's theorem (see Lemma \ref{lem:VHStheorem}) again, for any $\epsilon>0$, there exists $\delta>0$ such that  for any $A\in\mathcal{F}_j$ if $Q_j(A)<\d$ then $\l^j_{M,N}(A)<\e$ for all $M,N\in\mathbb{Z}_+$. Thus, it is easy to check that $V_j$ is continuous, and hence $V_j$ is an upper probability. From the definition of $V_j$, we have that $V_j$ is $T$-invariant. Then $(A_j,\mathcal{F}_j,V_j,T)$ is an UPPS. Let 
\begin{equation}\label{eq:21.24}
	V=\max_{1\le j\le n}V_j. 
\end{equation}
Then $(\O,\mathcal{F},V,T)$ is an UPPS.
We have the following properties.
\begin{prop}\label{prop:1}
	Let $(\O,\mathcal{F},P)$ be a probability space and $T:\O\to \O$ be a measurable invertible transformation. Suppose that $(\O,\mathcal{F},P,T)$ satisfies the finite irreducible invariant set partition assumption. Assume that for each $j=1,2,\ldots,n$, the limit
	\begin{equation}\label{eq:12.41}
		Q_j:=\lim_{N\to\infty}\frac{1}{N}\sum_{i=0}^{N-1}P_j\circ T^{-i}
	\end{equation}
	exists in the sense of \eqref{eq:5.1b}. Thus, for upper probabilities $V_j$, $V$ defined as \eqref{eq:8=} and \eqref{eq:21.24},  the following four statements hold:
	\begin{enumerate}
		\item[(i)] for each $j=1,2,\ldots,n$, $(A_j,\mathcal{F}_j,Q_j,T)$ is an ergodic measure-preserving system;
		\medskip
		\item [(ii)]for each $j=1,2,\ldots,n$, $(A_j,\mathcal{F}_j,V_j,T)$ is an ergodic UPPS;
		\medskip
		\item [(iii)]$(\O,\mathcal{F},V,T)$ is of finite ergodic components;
		\medskip
		\item [(iv)]$V(\mathcal{I})=\{0,1\}$.
	\end{enumerate}
\end{prop}
\begin{proof}
	We finish the proof by  proving (i), (i) $\Rightarrow$ (ii), (ii) $\Rightarrow$ (iii) and (iii) $\Rightarrow$ (iv).
	
	Firstly, we prove (i). Indeed, for any $j\in\{1,2,\ldots,n\}$ and $P_j:= \frac{P(\cdot\cap A_j)}{P(A_j)}$, it is easy to see that $P_j|_{\mathcal{I}}=Q_j|_{\mathcal{I}}$. By  finite irreducible invariant set partition assumption, for any $A\in\mathcal{I}\cap\mathcal{F}_j$, $P_j(A)\in\{0,1\}$, and hence $Q_j(A)\in\{0,1\}$. Thus, $Q_j$ is ergodic on $(A_j,\mathcal{F}_j)$ with respect to $T$, as $Q_j$ is $T$-invariant.

	(i) $\Rightarrow$ (ii) is proved in \cite[Remark 4.6]{FHLZ}.  
	
	(ii) $\Rightarrow$ (iii) is from the definition of finite ergodic components.
	
	(iii) $\Rightarrow$ (iv) follows from Theorem \ref{thm:1}.
\end{proof}

The following result is a version of Birkhoff's ergodic theorem for non-invariant probabilities.
\begin{theorem}\label{thm:noninvariant}
	Let $(\O,\mathcal{F},P)$ be a probability space and $T:\O\to \O$ be a measurable invertible transformation. If the limit \eqref{eq:20.19} exits, then the following statements are equivalent:
	\begin{enumerate}
		\item [(i)] $\#(P(\mathcal{I}))<\infty$;
		\medskip
		\item [(ii)] $(\O,\mathcal{F},P,T)$ satisfies the finite irreducible invariant set partition assumption; %for some $n\in\mathbb{N}$ and the limits in \eqref{eq:12.41} exist;
		\medskip
		\item [(iii)]there exist $Q_1,\ldots,Q_n\in\mathcal{M}^e(T)$ and  a finite measurable partition $\a=\{A_1,\ldots,A_n\}$ with $P(A_i)>0$, $i=1,2,\ldots,n$, such that for any $f\in B(\O,\mathcal{F})$, 
		\begin{equation}\label{eq:12.38}
			\lim_{N\to\infty}\frac{1}{N}\sum_{i=0}^{N-1}f\circ T^i=\sum_{j=1}^{n}\left(\int fdQ_j\right)\cdot1_{A_j},  ~~P\text{-a.s.}
		\end{equation}
		\item [(iv)]there exist  $Q_1,\ldots,Q_n\in\mathcal{ M}^e(T)$  and  a finite  measurable partition $\a=\{A_1,\ldots,A_n\}
		$ with $P(A_i)>0$, $i=1,2,\ldots,n$ such that 
		\begin{equation}\label{eq:15.54}
			\lim_{N\to\infty}\frac{1}{N}\sum_{i=0}^{N-1}P(B\cap T^{-i}C)=\sum_{j=1}^nQ_j(C)P(A_j\cap B)\text{ for any }B,C\in\mathcal{F}.
		\end{equation}
	\end{enumerate}
\end{theorem}
\begin{proof}
	The equivalence of (i) and (ii) is from Lemma \ref{lem:good partition}.	
	
	(ii) $\Rightarrow$ (iii) is from Proposition \ref{prop:1} and Theorem \ref{thm:Birkhoff's ergodic}.
	
	(iii) $\Rightarrow$ (i). For any $A\in\mathcal{I}$, letting 
	$f=1_A$ in \eqref{eq:12.38}, we have that 
	\[1_A=\sum_{j=1}^nQ_j(A)1_{A_j},\text{ }P\text{-a.s.}\]
	Integrating both side with respect to $P$, we have that
	\[P(A)=\sum_{j=1}^nQ_j(A)P(A_j).\]
	Thus, 
	\[P(\mathcal{I})\subset \sum_{j=1}^nQ_j(\mathcal{I})P(A_j):=\{a\in[0,1]:a=\sum_{j=1}^nQ_j(A)P(A_j),\text{ for some }A\in\mathcal{I}\}.\]

	Since $Q_j$ is ergodic, $Q_j(\mathcal{I})\in\{0,1\}$, for $j=1,2,\ldots,n$, it follows that $\#(P(\mathcal{I}))\le 2^n<\infty$.
	
	(ii) $\Rightarrow$ (iv) is from Proposition \ref{prop:1} and Theorem \ref{thm:ergodic via independece of core}.
	
	(iv) $\Rightarrow$ (i). For any $A\in\mathcal{I}$, letting 
	$B=\O$  and $C=A$ in \eqref{eq:15.54}, we have that 
	\[P(A)=\sum_{j=1}^nQ_j(A)P(A_j),\]
	which, by the same argument of ``(iii) $\Rightarrow$ (i)'', implies (i). 
\end{proof}

\begin{rem}
	In the end, note that for any $A\in\mathcal{F}$, $\cup_{i=-\infty}^\infty T^{-i} A\in\mathcal{I}$, and hence we have 
	\[\mathcal{I}=\{\cup_{i=-\infty}^\infty T^{-i}A:A\in\mathcal{F}\}.\]
	Thus, the condition $\#(P(\mathcal{I}))<\infty$ is equivalent to the condition that $\#(P(\{\cup_{i=-\infty}^\infty T^{-i}A:A\in\mathcal{F}\})<\infty$. This  facilitates	a method to check Assumption $(*)$.
\end{rem}

%\begin{acks}[Acknowledgments]
%	The authors would like to thank the anonymous referees, an Associate
%	Editor and the Editor for their constructive comments that improved the
%	quality of this paper.
%\end{acks}

\begin{funding}
The second author was partially supported by NNSF of China (12090012, 12090010, 12031019).
%The third author was supported by CSC of China (No. 202206340035), and partially supported by NNSF of China (12090012, 12090010).
The fourth author was supported  by the Royal Society Newton Fund grant (ref. NIF \textbackslash R1\textbackslash 221003) and an EPSRC Established Career Fellowship  
(ref. EP/S005293/2).
\end{funding}

\end{document}